\documentclass[letterpaper, reqno, 11pt]{article}
\usepackage[margin=1.0in]{geometry}
\usepackage{color,latexsym,amsmath,amssymb, amsthm, graphicx,subfigure,revsymb, enumerate,xcolor, mathtools, enumitem}
\usepackage[percent]{overpic}

\numberwithin{equation}{section}

\newcommand{\RR}{\mathbb{R}}

\newtheorem{thm}{Theorem}[section]
\newtheorem{lem}[thm]{Lemma}

\newtheorem{cor}[thm]{Corollary}

\theoremstyle{remark}

\newtheorem{rem}{Remark}[section]

\title{On the number of pairwise touching cylinders in $\mathbb{R}^d$}

\author{
Jozsef Solymosi\thanks{
Department of Mathematics, University of British Columbia, Vancouver, Canada, and {O}buda University, Budapest, Hungary.
\texttt{solymosi@math.ubc.ca}}
\and
Joshua Zahl\thanks{
Chern Institute of Mathematics, Nankai University, Tianjin, China, and 
Department of Mathematics, University of British Columbia, Vancouver, Canada. \texttt{jzahl@nankai.edu.cn}
}}

\date{\today}

\begin{document}

\maketitle

\begin{abstract}
John E. Littlewood posted the question {\em ``Is it possible in 3-space for seven infinite circular cylinders of unit radius each to touch all the others? Seven is the number suggested by counting constants.''} Boz\'oki, Lee, and R\'onyai constructed a configuration of 7 mutually touching unit cylinders. The best-known upper bounds show that at most 10 unit cylinders in $\mathbb{R}^3$ can mutually touch. We consider this problem in higher dimensions, and obtain exponential (in $d$) upper bounds on the number of mutually touching cylinders in $\mathbb{R}^d$. Our method is fairly flexible, and it makes use of the fact that cylinder touching can be expressed as a combination of polynomial equalities and non-equalities.
\end{abstract}

\section{Introduction}
In Problem 7 in his problem book \cite{Littlewood1968}, John E. Littlewood posted the question {\em ``Is it possible in 3-space for seven infinite circular cylinders of unit radius each to touch all the others? Seven is the number suggested by counting constants.''} (see also in \cite{BrassMoserPach2005} with some related questions). Andr\'as Bezdek \cite{Bezdek2005} proved that in $\mathbb{R}^3$ the maximum number of pairwise touching cylinders (of arbitrary radii) is at most 24. In the case of cylinders of unit radius, this was subsequently improved to 18 by Koizumi \cite{Koizumi2025}, and then to 10 by Dillon, Koizumi, and Luo \cite{DillonKoizumiLuo2025}. In the other direction, 
Boz\'oki, Lee, and R\'onyai \cite{BozokiLeeRonyai2017} constructed an arrangement of 7 mutually touching unit cylinders. For non-equal radii, there are numerical examples for 9 mutually touching cylinders in \cite{AlphaEvolve}.

In this paper, we consider the problem of bounding the number of pairwise touching cylinders (of unit radius or of arbitrary radius) in $\RR^d$. The parameter space of $n$-tuples of lines in $\RR^d$ has dimension $(2d-2)n$, while the requirement that the corresponding unit cylinders touch imposes $\binom{n}{2}$ constraints. The image of a configuration of touching cylinders under a rigid transformation is also a configuration of touching cylinders. Thus by counting parameters, one might expect to find $n$ touching unit cylinders in $\RR^d$ provided $(2d-2)n \geq \binom{n}{2} + \binom{d+1}{2}$. In \cite{DillonKoizumiLuo2025}, Dillon, Koizumi, and Luo construct an arrangement of $2d-2$ pairwise touching unit cylinders in $\RR^d$. Prior to this work, it was not known if the maximum number of touching unit cylinders in $\RR^d$ for $d\geq 4$ is finite. Our main results are the following:

\begin{thm}\label{unitDistanceLines}
Let $C$ be a set of mutually touching unit cylinders in $\RR^d$. Then $|C|\leq 4d\cdot 7^{2d-3}$.
\end{thm}
\begin{thm}\label{unitDistanceCylindersLines}
Let $C$ be a set of mutually touching cylinders (with arbitrary radii) in $\RR^d$. Then $|C|\leq 20(d+1)\cdot 7^{2d-2}$.
\end{thm}

Theorem \ref{unitDistanceLines} is equivalent to the problem of bounding the maximum number of lines in $\mathbb{R}^d$ with pairwise unit distances. It sounds similar to the classical 2-distance set estimates for points by Larman, Rogers, and Seidel \cite{LarmanRogersSeidel1977} and by Blukhuis \cite{Blokhuis1988}, but there are significant differences. There is no single polynomial vanishing on the parameters of two lines if and only if their distance is one. We have to restrict our approach to real spaces, where cylinder-touching can be expressed as a combination of polynomial equalities and non-equalities. As far as we know, this is a novel application of the polynomial method.


\section{An algebraic expression for touching cylinders}\label{algExpressionForTangencySection}
In what follows, we will identify points in $\RR^{2d-2}$ with lines in $\RR^d$. For $x = (a,b)\in\RR^{d-1}\times\RR^{d-1}$, we define the line $L_{x}(t) = \{(t, at+b)\colon t\in\RR\}$. Every line that is not parallel to the $\{x_1=0\}$ plane can be expressed in this form. A pair of lines corresponding to the points $x=(a,b)$ and $y=(c,d)$ are parallel if and only if $a=c$. 

Let $x=(a,b)$ and $y=(c,d)$. We wish to compute the squared distance between the corresponding lines $L_x$ and $L_{y}$. This will be a rational function of $x$ and $y$. It can be expressed as the ratio of the Gram determinant of the displacement and directions to the Gram determinant of the directions.

Define
\[
P(x,y) = \det \begin{pmatrix}
            |b-d|^2 		& 	(b-d)\cdot a  	& 	(b-d)\cdot c  \\
             a \cdot (b-d) & 	1 + \|a\|^2 	& 	1 +  a\cdot c   \\
            c \cdot (b-d) & 	1 + c\cdot a 	& 	1 + \|c\|^2
        \end{pmatrix},\quad
        Q(x,y) = \det \begin{pmatrix}
            1 + \|a\|^2 & 1 + a\cdot c \\
            1 +  c \cdot a & 1 + \|c\|^2
        \end{pmatrix}.
\]
Then 
\begin{equation}
    (\operatorname{dist}(L_{x}, L_{y}))^2 = P(x,y) / Q(x,y).
\end{equation}
For $x$ fixed, $P(x, y)$ and $Q(x, y)$ are polynomials of degrees 4 and 2, respectively, in $y$. Similarly with the roles of $x$ and $y$ reversed. Observe that $Q(x,y)=0$ if and only if $a=c$, i.e.~the lines $L_x$ and $L_y$ are parallel.
For $x,y\in\RR^{2d-2}$, define 
\begin{equation}\label{defnF}
F(x,y) = P(x,y) - 4Q(x,y).
\end{equation}
Similarly, for $w=(x,r),\ v=(y, s) \in \RR^{2d-2}\times \RR$, define 
\begin{equation}\label{defnG}
G(w,v)  = P(x,y) - (r+s)^2Q(x,y).
\end{equation} 
We record several observations. In what follows, we adopt the notation $w = (x,r)$ and $v = (y,s)$.
\begin{itemize}
	\item If the unit cylinders with coaxial lines $L_x$ and $L_y$ touch, then $F(x,y) = 0$. If the cylinder with coaxial line $L_x$ and radius $r$ touches the cylinder with coaxial line $y$ and radius $s$, then $G(w,v) = 0$.
	\item Suppose the lines $L_x$ and $L_{y}$ are not parallel. If $F(x,y)=0$, then the unit cylinders with coaxial lines $L_x$ and $L_y$ touch. If $G(w,v)=0$ and $r,s>0$, then the cylinder with coaxial line $L_x$ and radius $r$ touches the cylinder with coaxial line $L_y$ and radius $s$.
	\item For $x$ (resp.~$y$) fixed, $F(x,y)$ is a polynomial in $y$ (resp.~$x$) of degree 4. For $w$ (resp.~$v$) fixed, $G(w,v)$ is a polynomial in $v$ (resp.~$w$) of degree 4. 
\end{itemize}

\section{The restricted Zariski closure}
Let $H(x,y)\colon \RR^{k}\times\RR^k\to\RR$ be a real polynomial. Define $H_x(y) = H(x,y)$ and define $H^y(x) = H(x,y)$. Thus $Z(H_x)=\{y\in \RR^k\colon H(x,y) = 0\}$ and similarly for $Z(H^y)$. 

For $X\subset\RR^k$ and $H$ as above, define
\begin{equation}\label{defnXStar}
\overline X^H = \bigcap_{\substack{x\in\RR^k\\ X\subset Z(H_x)}}\!\!\!Z(H_x)\ \ \cap\bigcap_{\substack{y\in\RR^k\\ X\subset Z(H^x)}}\!\!\!Z(H^y),
\end{equation}
where by convention we define $\overline X^H=\RR^k$ if there are no sets in the above intersection. Note that $\overline X$ is closed (in both the Euclidean and Zariski topology) and that $X\subset \overline X\subset \overline X^H$, where $\overline X$ denotes the Zariski closure of $X$. We call $\overline X^H$ the (symmetric) restricted Zariski closure of $X$ with respect to $H$. 

\begin{lem}\label{HClosureLem}
Let $H(x,y)\colon \RR^{k}\times\RR^k\to\RR$ be a real polynomial, let $X\subset\RR^k$ be non-empty, and suppose $H(x,y)=0$ for all $x,y\in X$. Then $H(x,y)=0$ for all $x,y\in \overline X^H$.
\end{lem}
\begin{proof} 
To streamline notation, define $X^* = \overline X^H$. Let $x\in X$. Since $H(x,y)=0$ for all $y\in X$, we have $X\subset Z(H_x)$, and hence $X^*\subset Z(H_x)$. This means that for all $x\in X$ and all $y^*\in X^*$, we have $H(x,y^*) = 0$. Fix $y^*\in X^*$. Then for all $x\in X$ we have $H(x,y^*) = 0$, i.e. $X\subset Z(H^{y^*})$, and thus $X^*\subset Z(H^{y^*})$. Since $y^*\in X^*$ was arbitrary, we conclude that $H(x^*,y^*)=0$ for all $x^*,y^*\in X^*$.
\end{proof}
\begin{rem}
For some applications, we might only know that $H(x,y) = 0$ for \emph{distinct} $x,y\in X$. In this case, we can apply Lemma \ref{HClosureLem} to the polynomial $\tilde H(x,y) = H(x,y)\Vert x-y\Vert^2$. 
\end{rem}

\section{Milnor-Thom and its consequences}
The Milnor-Thom theorem \cite{Milnor1964, Thom1965} bounds the sum of the Betti numbers of a real algebraic variety. The following formulation is Theorem 11.5.3 from \cite{BCR}.
\begin{thm}\label{milnorThomThm}
Let $Z \subset\RR^k$ be a real algebraic variety defined by polynomials of degree at most $D$. Then the sum of the Betti numbers of $Z$ is at most $D(2D-1)^{k-1}$.
\end{thm}
Note that Theorem \ref{milnorThomThm} does not (a priori) require that $Z$ be defined by a \emph{finite} set of polynomials, though by Hilbert's basis theorem we can always reduce to this case.
Recall that the zeroth Betti number $b_0(Z)$ counts the number of Euclidean connected components of $Z$. We will record what happens when we apply Theorem \ref{milnorThomThm} to the set $\bar X^F$ defined in \eqref{defnXStar}, when $F\colon\RR^{2d-2}\times\RR^{2d-2}$ is the function from \eqref{defnF}.

%
\begin{cor}\label{consequenceOfMilnorThom}
Let $X\subset\RR^{2d-2}$. Then $\overline X^F$ is a union of at most $4\cdot 7^{2d-3}$ Euclidean connected components.
\end{cor}

A similar result holds for $Y\subset\RR^{2d-1}$ and $\overline Y^{G}$. However, even if (for example) all of the cylinders in $Y$ have radius $\geq 1$, the set $\overline Y^{G}$ may contain points corresponding to cylinders with zero, or negative, radii. In order to restrict attention to cylinders with radius at least one, we will intersect the set  $\overline Y^{G}$ with a half-space of the form $\{(x,r)\in\RR^{2d-1}\times\RR\colon r \geq 1\}$. The following result from \cite{BPR} allows us to control the number of Euclidean connected components that results from this process.

\begin{thm}
Let $Z \subset\RR^k$ be a real algebraic variety defined by polynomials of degree at most $D$. Let $P\colon\RR^k\to\RR$ be a polynomial of degree at most $D$. Then the sum of the number of connected components of the three sets $Z\cap\{P<0\}$, $Z\cap\{P=0\}$ and $Z\cap\{P>0\}$ is at most $5D(2D-1)^{k-1}$.
\end{thm}
\begin{cor}\label{consequenceOfBPR}
Let $Y\subset\RR^{2d-1}$. Then $\overline Y^G \cap \{(x,r)\in\RR^{2d-1}\times\RR\colon r \geq 1\}$ is a union of at most $20\cdot 7^{2d-2}$ Euclidean connected components.
\end{cor}

\section{Bounding the number of pairwise touching cylinders}
We are now ready to prove Theorems \ref{unitDistanceLines} and \ref{unitDistanceCylindersLines}.
\begin{proof}[Proof of Theorem \ref{unitDistanceLines}]
Since at most $d$ unit spheres in $\RR^{d-1}$ can mutually touch, at most $d$ pairwise parallel unit cylinders in $\RR^d$ can mutually touch. Thus it suffices to prove that every set of mutually touching unit cylinders with distinct directions has cardinality at most $4\cdot 7^{2d-3}$. 

Suppose to the contrary that there existed a set $C$ of  mutually touching unit cylinders with distinct directions that has cardinality $4\cdot 7^{2d-3}+1$. Applying a rotation if needed, we can suppose that none of these cylinders are parallel to the $\{x_1=0\}$ plane. Let $X\subset\RR^{2d-2}$ denote the corresponding set of points, as described in Section \ref{algExpressionForTangencySection}. Let $F$ be as defined in \eqref{defnF}. By Corollary \ref{consequenceOfMilnorThom}, there exists a Euclidean connected component $U\subset \bar X^F$ that contains at least two points of $X$.

Let $\pi\colon\RR^{2d-2}\to\RR^{d-1}$ be the projection to the first $d+1$ coordinates (this corresponds to the direction of the line). We claim that if $x,y\in U$ with $\pi(x)\neq\pi(y)$, then $\Vert x-y\Vert \geq 2$. This is because $x,y\in U \subset\overline X^F$ implies (by Lemma \ref{HClosureLem}) that $F(x,y) = 0$, while $\pi(x)\neq \pi(y)$ implies that $L_x$ and $L_y$ are not parallel. This means that $\Vert x-y\Vert \geq \operatorname{dist}(L_x, L_y) = 2$. In particular, $\pi(U)$ is (at most) countable, since every $\geq 2$ separated subset of $\RR^k$ is countable. On the other hand, $\pi(U)$ is connected and contains at least two points, and thus is uncountable. This contradiction completes the proof.
\end{proof}

\begin{proof}[Proof of Theorem \ref{unitDistanceCylindersLines}]
The proof proceeds in a similar fashion to that of Theorem \ref{unitDistanceLines}. Since at most $d+1$ circles in $\RR^d$ can mutually touch, it suffices to prove that every set of mutually touching cylinders with distinct directions has cardinality at most $20\cdot 7^{2d-2}$. Suppose for contradiction that there exists such a set of cylinders with cardinality $20\cdot 7^{2d-2}+1$. After rescaling we may suppose that each cylinder has radius $\geq 1$. Let $Y\subset\RR^{2d-1}$ be the corresponding set of points (the final coordinate denotes the radius of the cylinder). We wish to obtain a contradiction. 

Arguing as above and using Corollary \ref{consequenceOfBPR} in place of Corollary \ref{consequenceOfMilnorThom}, we obtain a Euclidean connected set $U\subset \bar Y^G \cap \{(x,r)\in\RR^{2d-1}\times\RR\colon r \geq 1\}$ that contains at least two points of $Y$, and thus $\pi(U)$ is uncountable. 
On the other hand, if $(x,r),\ (y,s)\in U$ and $\pi(x,r)\neq\pi(y,s)$, then $\Vert x-y\Vert \geq \operatorname{dist}(L_x, L_y) \geq (r+s)\geq 2$. This means that $\pi(U)$ must be (at most) countable. This is our desired contradiction.
\end{proof}

\section{Acknowledgements}
The first author's research was supported by an NSERC Discovery grant and by the National Research Development and Innovation Office
of Hungary, NKFIH, Grants no. KKP133819 and Excellence 151341. The second author's research was supported by a NSERC Discovery grant, the Nankai Zhide Foundation, and the Fundamental Research Funds for the Central Universities (Project No. 100-63253272).

\end{document}